\documentclass[12pt]{amsart}
\usepackage{enumitem, verbatim}
\usepackage{color}
\usepackage{appendix}
\usepackage{amsmath, amssymb, amsthm}
\usepackage{mathrsfs}

\usepackage[margin=3 cm,heightrounded=true,centering]{geometry}
\usepackage{graphicx}
\usepackage{wrapfig}
\usepackage{cancel}
\newtheorem{theorem}{Theorem}[section]
\newtheorem{proposition}[theorem]{Proposition}
\newtheorem{lemma}[theorem]{Lemma}
\newtheorem{remark}[theorem]{Remark}
\newtheorem{corollary}[theorem]{Corollary}

\DeclareMathOperator{\dv}{div}

\DeclareMathOperator{\ric}{Ric}
\DeclareMathOperator{\weyl}{Weyl}

\DeclareMathOperator{\hess}{Hess}
\DeclareMathOperator{\vol}{vol}

\DeclareMathOperator{\SU}{SU}

\DeclareMathOperator{\divergence}{div}

\newcommand{\riemuudu}[4]{R_{\phantom{#1 #2} #3}^{#1 #2 \phantom{#3} #4}}

\title{Rigidity of quasi-Einstein metrics: The incompressible case}

\author{Eric Bahuaud}
\address[Eric Bahuaud]{Department of Mathematics, Seattle University, Seattle, WA 98122, USA}
\email{bahuaude@seattleu.edu}

\author{Sharmila Gunasekaran}
\address[Sharmila Gunasekaran]{The Fields Institute for Research in Mathematical Sciences, 222 College St, Toronto ON, Canada M5T 3J1}
\email{gunasek1@ualberta.ca}

\author{Hari K Kunduri}
\address[Hari K Kunduri]{Department of Mathematics and Statistics and Department of Physics and Astronomy, McMaster University, Hamilton ON, Canada L8S 4K1}
\email{kundurih@mcmaster.ca}

\author{Eric Woolgar}
\address[Eric Woolgar]{Department of Mathematical and Statistical Sciences and Theoretical Physics Institute, University of Alberta, Edmonton, Alberta, Canada T6G 2G1}
\email{ewoolgar@ualberta.ca}

\begin{document}

\maketitle

\begin{abstract}
As part of a programme to classify quasi-Einstein metrics $(M,g,X)$ on closed manifolds and near-horizon geometries of extreme black holes, we study such spaces when the vector field $X$ is divergence-free but not identically zero. This condition is satisfied by left-invariant quasi-Einstein metrics on compact homogeneous spaces (including the near-horizon geometry of an extreme Myers-Perry black hole with equal angular momenta in two distinct planes), and on certain bundles over K\"ahler-Einstein manifolds. We find that these spaces exhibit a mild form of rigidity: they always admit a one-parameter group of isometries generated by $X$. Further geometrical and topological restrictions are also obtained.
\end{abstract}

\section{Introduction}
\setcounter{equation}{0}

\noindent One of the most important problems in geometric analysis is to find solutions of the Einstein equation
\begin{equation}
\label{eq1.1}
\ric = \lambda g,
\end{equation}
where $\lambda$ is a constant and $g$ is a complete Riemannian metric on a manifold $M$ or perhaps a Lorentzian metric on a spacetime with suitable completeness and causality conditions. Here $\ric$ is the Ricci tensor of $g$. In the course of solving this problem, one often deals with special cases that induce geometric equations on submanifolds. In many of these cases, an equation of the form
\begin{equation}
\label{eq1.2}
\ric +\frac12 \pounds_X g -\frac{1}{m} X\otimes X = \lambda g
\end{equation}
arises, where $X$ denotes a 1-form (in a common abuse of notation, we also use $X$ to denote the metric-dual vector field, as in the Lie derivative term here), and $m$ is a constant. One such example arises when $g$ is a warped product metric, in which case $X$ is a gradient vector field and $m$ is a positive integer. Another case is the near horizon limit, which arises on vacuum stationary spacetimes that contain a so-called extreme black hole. Then equation \eqref{eq1.2} with $m=2$ describes the geometry of a degenerate Killing horizon (coinciding with the black hole's event horizon) on which a Killing vector field that is timelike immediately outside the horizon becomes null. The horizon is a null hypersurface in spacetime ruled by integral curves of the Killing field. The degeneracy condition implies that these integral curves are affinely parametrized null geodesics, called the null generators of the horizon. The metric $g$ is a Riemannian metric giving the inner product on tangent spaces to this hypersurface modulo the null  Killing vector field tangent to the generating curves. We may take $g$ to be a metric on compact cross-sections $M$ of the horizon that are transversal to the null generators.

But equation \eqref{eq1.2} may be studied with $m$ any real number. The case of $m=0$ may be included if one takes $X$ to vanish, so this is used to denote the Einstein case. As well, $m=\pm \infty$ denotes the case where $\ric+\frac12 \pounds_X g =\lambda g$, which is the \emph{Ricci soliton equation}. The case of $0<m<\infty$ has been widely studied, usually with $X$ taken to be a gradient $X=\nabla f$. Some important works of this nature include for example \cite{KK}, \cite{Case}, and \cite{CSW}, but some authors also consider greater generality, allowing the constants $m$ and $\lambda$ to be replaced by functions (see, e.g., \cite{FFT}).

The case of near horizon geometries, however, motivates the study of the case of \eqref{eq1.2} when $X$ is not necessarily gradient. In their work on near horizon geometries, Kunduri and Lucietti \cite{KL} noted that the 1-form $X$ can be decomposed into a closed piece and a co-closed piece. In our previous work \cite{BGKW}, we studied the case where $X$ is closed, which obviously contains gradient vector fields as a subcase. When $X$ is closed, the quasi-Einstein manifold $(M,g,X)$ with $m=2$ is called a \emph{static near horizon geometry}, since static extreme black holes produce degenerate Killing horizons with this geometry (as do some non-static extreme black holes). In our work we allowed $m>0$ to be arbitrary; we did not require $m=2$. We obtained a structure result which improved and generalized a previous result of Chru\'sciel, Reall, and Tod \cite{CRT}. Soon thereafter, Wylie \cite{Wylie} used our results to give a satisfactory rigidity theorem for the static case.

At the opposite end of things, there is the case where $X$ is co-closed, i.e., divergence-free. We call this the \emph{incompressible} case. Left-invariant quasi-Einstein metrics on homogeneous spaces obey this condition. Indeed, they obey $\pounds_X g=0$, from which we easily obtain $\divergence X=0$ by taking the trace. An important near horizon geometry which obeys this condition is the 3-dimensional near horizon geometry of a 5-dimensional extreme Myers-Perry black hole with equal angular momenta in two different 2-planes. The near horizon geometry of this is a Berger sphere; i.e., $g$ is a certain left-invariant metric on $\SU(2)$. A near horizon geometry is an $m=2$ solution of \eqref{eq1.2}, but since $\pounds_X g =0$ it is easy to extend this example to any $m>0$ by rescaling $X$. Examples that are not homogeneous spaces are also known; see the Appendix for explicit examples.

It is natural to ask whether there is some rigidity in the incompressible case. We do not expect as much as in the static case, since we do not require $X$ to be hypersurface-orthogonal (and in the above examples, it is not). But there is considerable rigidity nonetheless, as our main theorem shows.

In two dimensions, since $\ric = (\mathrm{R}/2) g $, if $\pounds_X g =0$, then equation \eqref{eq1.2} can only be satisfied if $X=0$.  Thus our main result is of interest when $n \geq 3$.

\begin{theorem}\label{theorem1.1}
Let $m\neq 0, -2$ and $\lambda$ be constants. For $n\geq 3$, let $M^n$ be a closed connected manifold such that $(M,g,X)$ obeys \eqref{eq1.2}, where $X$ is a vector field that is not identically zero and obeys $\divergence X=0$. Then $X$ obeys Killing's equation $\pounds_X g=0$. Both $X$ and $dX$ have constant norm. The integral curves of $X$ are complete geodesics.
\end{theorem}

\begin{remark}\label{remark1.2}
The exceptional case of $m=-2$ appears to be an artifact of our method of proof, but we have not found a way to remove it. However, Proposition \ref{proposition2.4} shows that solutions of \eqref{eq1.2} on closed manifolds with $m<0$ and $\divergence X=0$ are Einstein (i.e., $X$ vanishes identically) whenever $\lambda\le 0$.
\end{remark}

When $\divergence X=0$, the (constant) norms of $X$ and $dX$ are related in a simple way. If $|X|^2=:c^2$, the relation is
\begin{equation}
\label{eq1.3}
|dX|^2=4c^2 \left ( \lambda +\frac{c^2}{m}\right ).
\end{equation}
One interesting consequence is that for $\lambda<0$ and $m>0$, we see from this relation that $|X|^2\ge -m\lambda>0$. Then there is no curve of $\lambda<0$ quasi-Einstein metrics with $\divergence X=0$, passing through an Einstein metric $(M,g_0)$, such that $X$ is continuous along the family at $(M,g_0)$ (unless $X$ remains zero along the curve).

Theorem \ref{theorem1.1} imposes limits on both the topology and the geometry of quasi-Einstein metrics in the incompressible case when $X$ is not identically zero. We now list such results.

\begin{corollary}\label{corollary1.3}
Under the assumptions of Theorem \ref{theorem1.1}, $M$ has vanishing Euler characteristic. In particular, $M$ cannot be homeomorphic to an even-dimensional sphere.
\end{corollary}

\begin{corollary}\label{corollary1.4}
Under the assumptions of Theorem \ref{theorem1.1}, we have the following.
\begin{itemize}
\item [i)] The Ricci tensor has precisely two distinct eigenvectors at each point, with multiplicities $1$ and $n-1$, and the dimension of $M$ must be $n\ge 3$.
\item [ii)] The eigenvector of multiplicity $1$ has nonnegative eigenvalue. If this eigenvalue is zero, then $dX=0$ and if $m>0$ (equivalently,\footnote
{These are equivalent because the eigenvalue can vanish only if $m\lambda<0$.}
if $\lambda<0$) then $(M,g)$ splits isometrically as $N\times {\mathbb S}^1$ where $N$ is negative Einstein.
\item [iii)] Conversely, if this eigenvalue is nonzero, then $dX$ is nonvanishing, so $(M,g,X)$ is not static and does not split isometrically along $X$.
\item [iv)] Let $C$ denote the Cotton tensor of $(M,g)$. Then $C=0 \Leftrightarrow dX=0 \Leftrightarrow |X|^2=-m\lambda$. Hence if either $(M,g)$ is locally conformally flat, or $n\ge 4$ and $(M,g)$ has harmonic Weyl curvature,\footnote
{We remind the reader that harmonic Weyl curvature means that the Weyl curvature tensor is divergence-free (on any index). Local conformal flatness for $n\ge 4$ implies zero Weyl curvature, which is a subcase of harmonic Weyl curvature.}
then $X$ is parallel, and if in addition $m>0$ (equivalently, if $\lambda<0$) then $(M,g)$ splits isometrically as $\Sigma\times {\mathbb S}^1$ where $\Sigma$ is a negative Einstein $(n-1)$-manifold.
\end{itemize}
\end{corollary}

While the eigenvalues of the Ricci tensor are constant, the eigenspaces are not parallelly propagated unless $dX=0$.

The next result collects some consequences when $n:=\dim M =4$.

\begin{corollary}\label{corollary1.5}
Under the assumptions of Theorem \ref{theorem1.1}, if $\dim M=4$ then the Weyl tensor of $(M,g,X)$ obeys
\begin{equation}
\label{eq1.4}
\int_M \left \vert \weyl \right \vert^2 dV = \frac4{3m^2} \left [ \left ( |X|^2\right )^2 -m\lambda |X|^2 -2m^2\lambda^2\right ]^2 \vol(M).
\end{equation}
Then
\begin{itemize}
\item [i)] if $m\lambda>0$ then $|X|^2\ge 2m\lambda$ and $(M,g)$ is conformally flat if and only if $|X|^2= 2m\lambda$.
\item [ii)] if $\lambda=0$, then $\int_M \left \vert \weyl \right \vert^2 dV=\frac{4}{3m}\left ( |X|^2\right )^2\vol(M) >0$ and $(M,g)$ is not locally conformally flat.
\item [iii)] if $m\lambda<0$ then $|X|^2\ge -m\lambda$. Moreover, $(M,g)$ is locally conformally flat if and only if $|X|^2= -m\lambda$ and then it splits as $\Sigma\times {\mathbb S}^1$ where $\Sigma$ is compact hyperbolic.
\end{itemize}
In any case, $b_1(M)\ge 1$ and so $M$ is not simply connected
\end{corollary}

Theorem \ref{theorem1.1} is itself a corollary of a general integral identity for solutions of \eqref{eq1.2} on closed manifolds. This identity is given in the following proposition, which holds with no assumption on $\divergence X$.

\begin{proposition}\label{proposition1.6}
Let $(M,g,X)$ obey equation \eqref{eq1.2} with $m\neq 0, -2$ on a closed manifold $M$. Then
\begin{equation}
\label{eq1.5}
\int_M \left ( \divergence X -\frac{1}{m} |X|^2\right )^2 dV = \frac12 \int_M \left \vert \pounds_X g\right \vert^2 dV + \frac{1}{m^2}\int_M \left ( |X|^2\right )^2 dV.
\end{equation}
\end{proposition}

Several integral formulae for quasi-Einstein metrics appear in \cite{BR}. That paper deals mainly---though not entirely---with the gradient case $X=df$, so that $X$ is an exact form and hence closed. When $X$ is closed, greater rigidity can be expected; cf \cite{BGKW}. We note that in the limit $m\to \infty$, which is known as the Ricci soliton case, Perelman's theorem implies that a compact Ricci soliton is a gradient soliton so that $X=df$, and we obtain from \eqref{eq1.5} the formula
\begin{equation}
\label{eq1.6}
\int_M \left ( \Delta f\right )^2 dV =2\int_M |\hess f|^2 dV,
\end{equation}
which previously appeared in \cite[Theorem 2]{ABR}.


Equation \eqref{eq1.5} governs the $L^2$-norm of the symmetrized part of $\nabla X$. A second integral formula governing the $L^2$-norm of the antisymmetrized part can also be obtained, as the following proposition shows, again with no assumption on $\divergence X$.

\begin{proposition}\label{proposition1.7}
Let $(M,g,X)$ obey equation \eqref{eq1.2}with $m\neq 0, -2, -4$ on a closed manifold $M$. Let $dX(\partial_i,\partial_j):=\nabla_i X_j - \nabla_j X_i$. Then
\begin{equation}
\label{eq1.7}
\begin{split}
&\, \int_M \left ( \lambda +\frac{2}{m} |X|^2\right )^2 dV\\
=&\, \lambda^2\vol(M) +\int_M \left [ \left ( \frac12 -\frac1m \right )\left \vert \pounds_X g\right \vert^2 +\frac1m \left \vert dX\right \vert^2 +\left ( \frac4m -1\right ) (\divergence X)^2 \right ] dV .
\end{split}
\end{equation}

\end{proposition}

This formula, which reduces to equation \eqref{eq1.3} when $\divergence X=0$ and (concomitantly) $|X|^2=const$ and $|dX|^2=const$, has some interesting consequences. In the context of Theorem \ref{theorem1.1}, we have the following corollary, which does not assume that $\divergence X=0$ but instead assumes that $|X|^2=-m\lambda>0$.

\begin{corollary}\label{corollary1.8}
Let $M$ be a closed manifold, let $m$ and $\lambda$ be constants of opposite sign such that $(M,g,X)$ obeys \eqref{eq1.2} with $|X|^2=-m\lambda$.
\begin{itemize}
\item [i)] If $2<m\le 4$ then $X$ is parallel and thus static, and $(M,g)$ splits isometrically as $N\times {\mathbb S}^1$ where $N$ is negative Einstein.
\item [ii)] If $m=2$ then $X$ is closed and incompressible, and if $n=4$ then $(M,g)$ splits isometrically as $N\times {\mathbb S}^1$ where $N$ is negative Einstein.
\end{itemize}
\end{corollary}

This paper is organized as follows. The proof of Theorem \ref{theorem1.1} is based on Proposition \ref{proposition1.6}, so we prove Proposition \ref{proposition1.6} first. Its proof uses a calculation presented in subsections \ref{subsection2.1} and \ref{subsection2.2}. The proof of Proposition \ref{proposition1.6} is presented at the start of subsection \ref{subsection2.3}, where it is followed by the proof of Theorem \ref{theorem1.1} and the other main results. Appendix \ref{appendix} contains a discussion of nontrivial examples of quasi-Einstein metrics on closed manifolds that are not homogeneous spaces and for which $X$ is incompressible (i.e., $\divergence X=0$) but not closed (so $dX$ does not vanish).

Throughout this paper, we take $M$ to be a closed, connected, orientable manifold.

\subsection*{Acknowledgements} The research of EB was partially supported by a Simons Foundation Grant (\#426628, E. Bahuaud). The research of HK was supported by NSERC grant RGPIN--04887--2018. The research of EW was supported by NSERC grant RGPIN--2022--03440. EW thanks the Fields Institute for hospitality during a visit that facilitated the completion of this work.

\subsection*{Note added} As we were in the final stage of releasing this paper, we became aware of the work \cite{DL}, which shows that every $m=2$ quasi-Einstein metric with nongradient $X$ admits a nontrivial incompressible vector field and, in consequence, a nontrivial Killing vector field. That field does not always equal $X$ and may have isolated zeroes.

\subsection*{Data statement} Data sharing is not applicable to this article as no datasets were generated or analyzed during the current study.

\subsection*{Conflict of interest statement} The authors have no conflicts of interest.

\section{The proofs of the results}
\setcounter{equation}{0}

\noindent The proof is through a series of straightforward calculations which we use to obtain a differential equation governing $|X|^2$. The individual lemmata below are intermediate steps in the calculation.

\subsection{Preliminaries}\label{subsection2.1}

\noindent From equation \eqref{eq1.2} the contracted Bianchi identity gives
\begin{equation} \label{eq2.1}
\nabla^i R_{ij} = \frac{1}{2} \nabla_j R = \frac{1}{2m} \nabla_j \left (|X|^2\right )- \frac{1}{2} \nabla_j \left (\dv X\right ) .
\end{equation}
We will also need that
\begin{equation} \label{eq2.2}
\nabla_j \left ( |X|^2\right ) = 2 X^i \nabla_j X_i,
\end{equation}
and
\begin{equation}
\label{eq2.3}
 \Delta\left ( |X|^2\right ) = 2 |\nabla X|^2+ 2 \left \langle X, \Delta X \right \rangle .
\end{equation}
A useful variant on this will be
\begin{equation} \label{eq2.4}
\left \langle X, \nabla_X X \right \rangle = \frac{1}{2} \nabla_X \left ( |X|^2\right ) .
\end{equation}
We introduce the following notation for the symmetrized and antisymmetrized parts, respectively, of the covariant derivative of $X$.
\begin{equation}
\label{eq2.5}
\begin{split}
S_{ij} :=&\, \pounds_X g_{ij} = \nabla_i X_j +\nabla_j X_i , \\
A_{ij} :=&\, (dX)_{ij} =  \nabla_i X_j -\nabla_j X_i .
\end{split}
\end{equation}
It is easy to check that
\begin{equation}
\label{eq2.6}
\begin{split}
|\nabla X|^2 =&\, \frac{1}{4} \left (|S|^2 + |A|^2\right ) ,\\
\nabla^i X^j \nabla_j X_i =&\, \frac{1}{4} \left ( |S|^2 - |A|^2\right ) ,\\
\left \langle S,A\right \rangle =&\ 0.
\end{split}
\end{equation}

\subsection{Divergence equations}\label{subsection2.2}

\begin{lemma} \label{lemma2.1}
Solutions of \eqref{eq1.2} obey
\begin{equation}
\label{eq2.7}
\begin{split}
0=&\, \frac{1}{2} \Delta X_j  + \frac{\lambda}{2} X_j - \left( \frac{1}{m} + \frac{1}{4} \right) \nabla_X X_j + \frac{1}{2} \left( \frac{1}{m} - \frac{1}{4} \right)\nabla_j \left ( |X|^2\right )\\
&\, + \frac{1}{2m} X_j |X|^2  - \frac{1}{m} X_j \divergence X.
\end{split}
\end{equation}
\end{lemma}
\begin{proof}
Taking the divergence of \eqref{eq1.2}, we have
\begin{equation}
\label{eq2.8}
\nabla^i R_{ij} + \frac{1}{2} \Delta X_j + \frac{1}{2} \nabla_i\nabla_j X^i - \frac{1}{m} X_j \divergence X - \frac{1}{m}\nabla_X X_j = 0.
\end{equation}
We apply the contracted Bianchi identity and the Ricci identity to obtain
\begin{equation}
\label{eq2.9}
\frac{1}{2} \nabla_j R + \frac{1}{2} \Delta X_j  + \frac{1}{2} \nabla_j \left (\dv X \right ) + \frac{1}{2} R_{ij} X^i - \frac{1}{m} X_j \divergence X - \frac{1}{m} \nabla_X X_j = 0.
\end{equation}
This can be simplified using the trace of \eqref{eq1.2} and re-arranged to yield
\begin{equation}
\label{eq2.10}
 \frac{1}{2} \Delta X_j  + \frac{1}{2} R_{ij} X^i - \frac{1}{m} X_j \divergence X + \frac{1}{2m} \nabla_j \left (|X|^2 \right ) - \frac{1}{m} \nabla_X X_j = 0.
\end{equation}
Using \eqref{eq1.2} to replace $R_{ij}$, we get
\begin{equation}
\label{eq2.11}
\begin{split}
0=&\, \frac{1}{2} \Delta X_j  + \frac{\lambda}{2} X_j - \frac{1}{8}\nabla_j \left ( |X|^2\right ) - \frac{1}{4} \nabla_X X_j + \frac{1}{2m} X_j |X|^2  - \frac{1}{m} X_j \divergence X\\
&\, - \frac{1}{m} \nabla_X X_j + \frac{1}{2m}\nabla_j \left ( |X|^2\right ) .
\end{split}
\end{equation}
But this is \eqref{eq2.7}.
\end{proof}

An interesting corollary (though we will not use it in what follows) is the following.

\begin{corollary}\label{corollary2.2}
When $\divergence X=0$, $|X|^2=c^2=const$, and $\nabla_XX=0$, then $X$ obeys
\begin{equation}
\label{eq2.12}
\Delta X_i + \left ( \lambda +\frac{c^2}{m}\right ) X_i =0 .
\end{equation}
\end{corollary}
Of course, once we prove Theorem \ref{theorem1.1}, the assumptions $|X|^2=c^2=const$ and $\nabla_XX=0$ used above will follow from the assumption $\divergence X=0$.

We next obtain a differential equation governing $|X|^2$.

\begin{lemma}\label{lemma2.3}
Solutions of \eqref{eq1.2} satisfy
\begin{equation}
\label{eq2.13}
0= \Delta \left (|X|^2\right ) - \frac{1}{2} |S|^2 - \frac{1}{2} |A|^2 -  \nabla_X \left ( |X|^2 \right ) + 2 |X|^2 \left ( \lambda + \frac{1}{m} |X|^2 \right ) - \frac{4}{m} |X|^2 \divergence X .
\end{equation}
\end{lemma}
\begin{proof}
If we contract equation \eqref{eq2.7} against $X$, we obtain
\begin{equation}
\label{eq2.14}
\begin{split}
0=&\, \frac{1}{2} \left \langle X, \Delta X \right \rangle + \frac{\lambda}{2} |X|^2 - \frac12 \left( \frac{1}{m} + \frac{1}{4} \right) \nabla_X \left ( |X|^2\right ) \\
&\, + \frac{1}{2} \left( \frac{1}{m} - \frac{1}{4} \right)\nabla_X \left ( |X|^2\right ) + \frac{1}{2m} |X|^4  - \frac{1}{m} |X|^2 \divergence X,
\end{split}
\end{equation}
which can be rewritten as
\begin{equation}
\label{eq2.15}
0= \frac{1}{4} \Delta \left ( |X|^2\right ) - \frac{1}{2} |\nabla X|^2 + \frac{\lambda}{2} |X|^2 - \frac{1}{4}  \nabla_X \left ( |X|^2\right ) + \frac{1}{2m} |X|^4  - \frac{1}{m} |X|^2 \divergence X.
\end{equation}
Applying the formulas for $|\nabla X|$ in terms of $S$ and $A$, we obtain \eqref{eq2.13}.
\end{proof}

Before continuing with the main argument, we note the following fact, which shows that that quasi-Einstein metrics on closed manifolds with incompressible $X$ and $m\le 0$ exist only if $\lambda>0$.

\begin{proposition}\label{proposition2.4}
If $(M,g,X)$ obeys \eqref{eq1.2} on a closed manifold $(M,g)$ with $\divergence X = 0$, $\lambda\le 0$, and $m<0$, then $|X| \equiv 0$ and $(M,g)$ is Einstein.
\end{proposition}

\begin{proof}
If we integrate \eqref{eq2.15} over $M$ , we get
\begin{equation}
\label{eq2.16}
0= \int_M \left [ - \frac{1}{2}|\nabla X|^2 + \frac{\lambda}{2} |X|^2 + \frac{1}{2m} |X|^4  +\left ( \frac14 - \frac{1}{m}\right ) |X|^2 \divergence X\right ] dV .\\
\end{equation}
Setting $\divergence X=0$ in \eqref{eq2.16}, we see that if $\lambda\le 0$ and $m<0$ then the right-hand side is negative-definite unless $X$ vanishes identically.
\end{proof}

Returning to the main thrust, the divergence of equation \eqref{eq2.7} yields a differential equation governing $\divergence X$, which we will later use in conjunction with \eqref{eq2.13} governing $\divergence X$ to produce the main result.

\begin{lemma}\label{lemma2.5}
Solutions of \eqref{eq1.2} satisfy
\begin{equation}
\label{eq2.17}
\begin{split}
&\, -\Delta \dv X  + \left(\frac{4}{m} + 1\right) \nabla_X\left ( \divergence X \right ) - 2 \lambda \dv X - \frac{1}{m} |X|^2 \divergence X + \frac{2}{m} (\dv X)^2\\
=&\, \left( \frac{1}{m} - \frac{1}{4} \right) \Delta \left ( |X|^2\right )  - \left(\frac{3}{8} + \frac{1}{2m} \right) |S|^2 + \left(\frac{1}{2m} +  \frac{1}{8}\right) |A|^2 \\
&\, + \left(\frac{3}{m} + \frac{1}{4} \right) \nabla_X \left ( |X|^2\right )- 2\left( \frac{1}{m} + \frac{1}{4} \right) |X|^2 \left[  \lambda + \frac{1}{m} |X|^2 \right] .\\
\end{split}
\end{equation}
\end{lemma}
\begin{proof}

Taking the divergence of equation \eqref{eq2.7} and multiplying by $2$, we obtain
\begin{equation}
\label{eq2.18}
\begin{split}
0=&\, \nabla^j \left ( \Delta X_j\right )  + \lambda \dv X - 2\left( \frac{1}{m} + \frac{1}{4} \right) \nabla^jX^i \nabla_i X_j - 2\left ( \frac{1}{m} + \frac{1}{4} \right ) X^i \nabla_j \nabla_i X^j \\
&\, +  \left( \frac{1}{m} - \frac{1}{4} \right) \left ( \Delta |X|^2\right ) + \frac{1}{m} |X|^2 \divergence X + \frac{1}{m} \nabla_X \left ( |X|^2 \right )- \frac{2}{m} \nabla_X (\dv X)\\
&\, - \frac{2}{m} (\dv X)^2.
\end{split}
\end{equation}
The Ricci identity and equation \eqref{eq1.2} give
\begin{equation}
\label{eq2.19}
X^i \nabla_j \nabla_i X^j = \nabla_X\left ( \divergence X \right ) + \lambda |X|^2 + \frac{1}{m} |X|^4 - \frac{1}{2} \nabla_X \left ( |X|^2 \right ) .
\end{equation}
Next, we simplify $\nabla^j \left ( \Delta X_j\right )$ as follows.
\begin{equation}
\label{eq2.20}
\begin{split}
\nabla^j \left ( \Delta X_j\right ) =&\, \nabla^i\nabla_j\nabla_i X^j + \riemuudu{i}{j}{j}{k} \nabla_i X_k + \riemuudu{i}{j}{i}{k} \nabla_k X_j \\
=&\, \nabla^i\nabla_j\nabla_i X^j + R^{ki}\nabla_i X_k - R_{jk} \nabla^k X^j \\
=&\, \nabla^i\nabla_j\nabla_i X^j \\
=&\, \Delta \left ( \dv X \right ) + \nabla^i \left (R_{ij} X^j \right ) \\
=&\, \Delta \left ( \dv X \right ) + X^j\nabla^i R_{ij} + R_{ij} \nabla^i X^j \\
=&\, \Delta \left ( \dv X \right ) + \frac{1}{m} \nabla_X \left ( |X|^2\right ) - \frac{1}{2} |\nabla X|^2 - \frac{1}{2} \nabla^i X^j \nabla_j X_i + \lambda \dv X\\
&\, - \frac{1}{2} \nabla_X\left ( \divergence X \right ) .
\end{split}
\end{equation}
Now we rewrite equation \eqref{eq2.18} using \eqref{eq2.19} and \eqref{eq2.20}. We obtain
\begin{equation}
\label{eq2.21}
\begin{split}
0=&\, \Delta \left ( \dv X\right ) + \frac{1}{m} \nabla_X \left (|X|^2\right ) - \frac{1}{2} |\nabla X|^2 - \frac{1}{2} \nabla^i X^j\nabla_j X_i + \lambda \dv X\\
&\, - \frac{1}{2} \nabla_X\left ( \divergence X \right ) + \lambda \dv X - 2\left( \frac{1}{m} + \frac{1}{4} \right) \nabla^i X^j\nabla_j X_i\\
&\, - 2\left ( \frac{1}{m} + \frac{1}{4} \right) \left [ \nabla_X\left ( \divergence X \right )  + \lambda |X|^2 + \frac{1}{m} |X|^4- \frac{1}{2} \nabla_X \left ( |X|^2\right ) \right ]\\
&\, + \left ( \frac{1}{m} - \frac{1}{4} \right ) \Delta \left ( |X|^2\right ) + \frac{1}{m} |X|^2 \divergence X + \frac{1}{m} \nabla_X\left (|X|^2\right ) \\
&\, - \frac{2}{m} \nabla_X (\dv X) - \frac{2}{m} (\dv X)^2 .
\end{split}
\end{equation}
Re-organizing terms, this reads
\begin{equation}
\label{eq2.22}
\begin{split}
&\, \left ( \frac{1}{m} - \frac{1}{4} \right ) \Delta \left ( |X|^2\right ) - \frac{1}{2} |\nabla X|^2    - 2\left ( \frac{1}{m} + \frac{1}{2} \right ) \nabla^i X^j \nabla_j X_i \\
&\, + \left ( \frac{3}{m} + \frac{1}{4} \right ) \nabla_X \left ( |X|^2\right ) - 2\left ( \frac{1}{m} + \frac{1}{4} \right ) |X|^2 \left ( \lambda + \frac{1}{m} |X|^2 \right ) \\
=&\, -\Delta \dv X  + \left ( \frac{4}{m} + 1\right ) \nabla_X\left ( \divergence X \right ) - 2 \lambda \dv X - \frac{1}{m} |X|^2 \divergence X\\
&\, + \frac{2}{m} (\dv X)^2.
\end{split}
\end{equation}
It remains to replace $|\nabla X|^2$ and $\nabla^i X^j \nabla_j X_i$ with their expressions in terms of $S$ and $A$. Doing so, we obtain
\begin{equation}
\label{eq2.23}
\begin{split}
&\, \left( \frac{1}{m} - \frac{1}{4} \right) \Delta \left ( |X|^2\right ) - \left(\frac{3}{8} + \frac{1}{2m} \right) |S|^2 + \left(\frac{1}{2m} +  \frac{1}{8}\right) |A|^2 \\
&\, + \left(\frac{3}{m} + \frac{1}{4} \right) \nabla_X \left ( |X|^2\right ) - 2\left( \frac{1}{m} + \frac{1}{4} \right) |X|^2 \left ( \lambda + \frac{1}{m} |X|^2 \right ) \\
=&\, -\Delta \left ( \dv X \right ) + \left(\frac{4}{m} + 1\right) \nabla_X\left ( \divergence X \right ) - 2 \lambda \dv X - \frac{1}{m} |X|^2 \divergence X\\
&\, + \frac{2}{m} (\dv X)^2.
\end{split}
\end{equation}
But this is \eqref{eq2.17}. \end{proof}

\subsection{The proofs} \label{subsection2.3}

We begin with the proof of Proposition \ref{proposition1.6}, which follows straightforwardly from the results of the previous subsection. In turn, the proof of the main theorem follows directly from Proposition \ref{proposition1.6}.

\begin{proof}[Proof of Proposition \ref{proposition1.6}]
Multiply equation \eqref{eq2.13} by $(1/m + 1/4)$ and add the result to equation \eqref{eq2.17} to get
\begin{equation}
\label{eq2.24}
\begin{split}
&\, \frac{2}{m} \Delta \left ( |X|^2\right ) + \frac{2}{m} \nabla_X \left ( |X|^2\right ) - \left ( \frac{1}{m} + \frac{1}{2} \right ) \left \vert \pounds_X g\right \vert^2 \\
=&\, -\Delta \left ( \dv X \right )  + \left ( \frac{4}{m} + 1\right ) \nabla_X\left ( \divergence X \right ) - 2 \lambda \dv X  + \frac{4}{m^2} |X|^2 \divergence X\\
&\, + \frac{2}{m} (\dv X)^2.
\end{split}
\end{equation}
Integrating over the closed manifold $M$, we obtain
\begin{equation}
\label{eq2.25}
\begin{split}
&\, \int \left [ \frac{2}{m}|X|^2\divergence X +\left( \frac{1}{m} + \frac{1}{2} \right ) \left \vert \pounds_X g\right \vert^2 \right ] dV\\
=&\, \int_M \left [ \left ( \frac{2}{m} + 1\right ) (\divergence X)^2 -\frac{4}{m^2} |X|^2 \divergence X\right ] dV .
\end{split}
\end{equation}
This can be written as
\begin{equation}
\label{eq2.26}
0=\left ( 1+\frac{2}{m}\right ) \int_M \left [ \left ( \divergence X -\frac{2}{m} |X|^2 \right ) \divergence X -\frac12 \left \vert \pounds_X g\right \vert^2 \right ] dV.
\end{equation}
Since by assumption $m\neq-2$, we can drop the prefactor of $1+\frac{2}{m}$. We can also complete the square in the first term. Then
\begin{equation}
\label{eq2.27}
0=\int_M \left [ \left ( \divergence X -\frac{1}{m} |X|^2 \right )^2-\frac{1}{m^2} \left ( |X|^2 \right )^2
-\frac12 \left \vert \pounds_X g\right \vert^2 \right ] dV,
\end{equation}
which completes the proof.
\end{proof}

\begin{proof}[Proof of Theorem \ref{theorem1.1}]
First, it is immediately clear that if $\divergence X -\frac{1}{m} |X|^2=0$ pointwise, then $X$ must vanish pointwise [proof: integrate $\divergence X -\frac{1}{m} |X|^2=0$ over $M$] so $(M,g)$ is Einstein.\footnote
{The reader may notice that, strictly speaking, this contradicts the assumption that $m\neq 0$.}
Otherwise, setting $\divergence X=0$ in \eqref{eq1.5}, we obtain $\int_M \left \vert \pounds_X g\right \vert^2 dV=0$ and hence $\pounds_X g=0$ pointwise. Then \eqref{eq2.24} reduces to
\begin{equation}
\label{eq2.28}
0=\frac{2}{m} \Delta \left ( |X|^2\right ) + \frac{2}{m} \nabla_X \left ( |X|^2\right ).
\end{equation}
Multiplying by $-m|X|^2$ and integrating over $M$, we get that
\begin{equation}
\label{eq2.29}
2\int_M \left \vert \nabla \left ( |X|^2\right )\right \vert^2 dV =-\int_M \left ( |X|^2\right )^2 \divergence X dV =0.
\end{equation}
where we use that $\divergence X=0$. Hence $|X|=c=const$, and $c=0$ only if $X\equiv 0$ whence $(M,g)$ is Einstein. Then
\begin{equation}
\label{eq2.30}
0=\left \langle X,\pounds_X g\right \rangle =\nabla_X X +\frac12 \nabla \left ( |X|^2\right )= \nabla_X X ,
\end{equation}
so the integral curves of $X$ are (since $c\neq 0$) constant speed geodesics. Moreover, \eqref{eq2.13} reduces to
\begin{equation}
\label{eq2.31}
0=-|A|^2 + 4c^2\left ( \lambda +\frac{c^2}{m}\right ) ,
\end{equation}
and since $A=dX$ we recover equation \eqref{eq1.3}; in particular, $|dX|$ is constant.
\end{proof}

\begin{proof}[Proof of Corollary \ref{corollary1.3}]
By assumption the norm of $X$ is somewhere nonzero. Since the theorem asserts that the norm is constant, it is nowhere zero, so by the Poincar\'e-Hopf theorem the Euler characteristic must vanish. But even-dimensional spheres have Euler characteristic $2$.
\end{proof}

\begin{proof}[Proof of Corollary \ref{corollary1.4}]
Theorem \ref{theorem1.1} implies that $\pounds_X g=0$, so \eqref{eq1.2} reduces to
\begin{equation}
\label{eq2.32}
\ric=\lambda g +\frac{1}{m}X\otimes X.
\end{equation}

To prove (i), since $|X|^2=c^2$ where $c$ is a positive constant, the eigenvalues are $\mu_{n-1}:=\lambda$ (with multiplicity $n-1$) and $\mu_1:=\lambda+\frac{c^2}{m}$ (with multiplicity $1$). If $n=2$, the Ricci tensor obeys $\ric=\frac12 Rg$ and so cannot have two distinct eigenvalues at a point, so $n\ge 3$.

To prove (ii), observe that the eigenvalue $\mu_1$ has eigenvector $X$ which as we have noted is also a Killing vector. Then a simple application of a Bochner formula (see, e.g., \cite[Theorem 37, p 192]{Petersen}) yields
\begin{equation}
\label{eq2.33}
\begin{split}
0=&\, \frac12 \Delta \left ( |X|^2\right ) = |\nabla X|^2-\ric(X,X)= |\nabla X|^2-\lambda |X|^2-\frac{1}{m}\left ( |X|^2\right )^2\\
=&\,|\nabla X|^2-\lambda c^2 -\frac{c^4}{m}=|\nabla X|^2-c^2\mu_1.
\end{split}
\end{equation}
Then $\mu_1\ge 0$, and is zero if and only if $X$ is parallel, whence obviously $dX$ vanishes. If $\mu_1=0$ and  since by assumption $c^2=|X|^2\neq 0$ and $m>0$, then $\lambda<0$ and by the theorem of Wylie \cite{Wylie}, then $M$ splits as claimed.

To prove (iii), if $\mu_1>0$, then $X$ is not parallel; i.e., $\nabla X$ is nonzero. But $\pounds_X g=0$, so $dX$ is nonvanishing and $(M,g,X)$ cannot be static.

To prove (iv), we use \eqref{eq2.32} and the definition of the Cotton tensor to write
\begin{equation}
\label{eq2.34}
\begin{split}
C_{kij}:=& \frac{1}{(n-2)} \left [ \nabla_j \left ( R_{ik}-\frac{1}{2(n-1)}Rg_{ik}\right ) - \nabla_i \left ( R_{jk}-\frac{1}{2(n-1)}Rg_{jk}\right )\right ]\\
=&\, \frac{1}{m(n-2)}\left [ X_i\nabla_j X_k +X_k\nabla_jX_i-X_j\nabla_i X_k -X_k\nabla_iX_j\right ]\\
=&\, \frac{1}{2m(n-2)}\left [X_i \left ( \nabla_j X_k -\nabla_k X_j \right ) + X_j \left ( \nabla_k X_i -\nabla_i X_k \right ) \right . \\
&\, \left . -2X_k \left ( \nabla_i X_j -\nabla_j X_i \right )\right ]\\
=&\, \frac{1}{2m(n-2)}\left [X_i \left ( dX \right )_{jk} + X_j \left ( dX \right )_{ki}-2X_k \left ( dX \right )_{ij}\right ] ,
\end{split}
\end{equation}
where we have used from Theorem \ref{theorem1.1} that $|X|^2$ is constant (so the scalar curvature $R$ is constant). Thus $C=0$ whenever $dX=0$, which by \eqref{eq1.3} is equivalent to $|X|^2=-m\lambda$ (and occurs only when $m\lambda<0$; recall we exclude $X\equiv 0$ by assumption). To see the converse, we contract \eqref{eq2.34} against $X^k$ and use that $X^k\nabla_i X_k=\frac12 \nabla_i \left ( |X|^2\right )=0$ and $X^k \nabla_k X_i=0$. We get
\begin{equation}
\label{eq2.35}
X^k C_{kij} = -\frac{1}{m(n-2)}|X|^2 \left ( \nabla_i X_j -\nabla_j X_i \right ) = -\frac{1}{m(n-2)}|X|^2(dX)_{ij}.
\end{equation}
If $dX$ does not vanish then neither does the right-hand side, and so nor does $C$.

If $n=3$ and $(M,g)$ is locally conformally flat, then $C=0$. If $n\ge 4$ and $(M,g)$ has harmonic Weyl curvature, this means that the divergence of the Weyl tensor is zero, and a Bianchi identity then implies that $C=0$. Therefore, in either case we see from \eqref{eq2.35} that $dX=0$, so the manifold is \emph{static} in the terminology of \cite{BGKW}, and since we also have that $\pounds_X g=0$ then $\nabla X=0$. For $\lambda<0$ the splitting follows from the theorem of \cite{Wylie}.
\end{proof}

\begin{proof}[Proof of Corollary \ref{corollary1.5}]
The Euler characteristic $\chi(M)$ of a closed manifold $M$ is expressible as the integral over the manifold of the Pfaffian of the curvature 2-form. Writing the Pfaffian in terms of the Weyl, Ricci, and scalar curvature, the formula for a 4-manifold is
\begin{equation}
\label{eq2.36}
\frac{4\pi^2}{3}\chi(M)=\frac{1}{24} \int_M \left ( |\weyl|^2-2|\ric|^2+\frac23 R^2\right ) dV,
\end{equation}
where $|\weyl|^2 = W_{ijkl} W^{ijkl}$ is the norm of the Weyl tensor regarded as a $(0,4)$-tensor. A short computation using equation \eqref{eq2.32} yields
\begin{equation}
\label{eq2.37}
\begin{split}
-2|\ric|^2+\frac23 R^2 =&\, -2 \left \vert \lambda g +\frac1m X\otimes X \right \vert^2 +\frac23 \left ( 4\lambda +\frac1m |X|^2\right )^2\\
=&\, \frac83 \lambda^2 +\frac{4\lambda}{3m}|X|^2 -\frac{4}{3m^2} \left ( |X|^2\right )^2.
\end{split}
\end{equation}
Hence, using from Theorem \ref{theorem1.1} that $|X|^2=c^2=const$, we have
\begin{equation}
\label{eq2.38}
\frac{4\pi^2}{3}\chi(M)=\frac{1}{24} \int_M |\weyl|^2dV +\frac{1}{24}\left [\frac83 \lambda^2 +\frac{4\lambda}{3m}c^2 -\frac{4c^4}{3m^2} \right ] \vol(M) .
\end{equation}
But by Corollary \ref{corollary1.3} we have $\chi(M)=0$, so we recover \eqref{eq1.4}.

Setting $\weyl=0$ and $\lambda=0$ in \eqref{eq1.4} yields $c=0$. For $\lambda\neq 0$ but $\weyl=0$, we find a positive root given by
\begin{equation}
\label{eq2.39}
|X|^2\equiv c^2 = \frac{m\lambda}{2}+\frac{3|m\lambda|}{2}.
\end{equation}
Conversely, when $|X|^2$ takes the value given by \eqref{eq2.39}, the right-hand side of \eqref{eq1.4} vanishes and therefore $\weyl=0$ pointwise, so $(M,g)$ is locally conformally flat. The conditions in (i) that $|X|^2\ge 2m\lambda$ and in (iii) that $|X|^2\ge -m\lambda$ arise since the right-hand side of \eqref{eq1.4} cannot be negative. The splitting in (iii) arises because $|X|^2= -m\lambda$ implies that $dX=0$, and then the theorem of \cite{Wylie} gives the isometric splitting $M=N\times {\mathbb S}^1$ where $N$ is closed, negative Einstein, and 3-dimensional, hence compact hyperbolic.

Since $M$ is assumed compact, connected, and orientable, and since $M$ admits a nonvanishing vector field $X$, we have $0=\chi(M)=\sum_{k=0}^4 (-1)^kb_k(M)\ge 2(1-b_1(M))$ so $b_1(M)\ge 1$ and therefore $M$ cannot be simply connected (for any $\lambda$ and any value of $c>0$).

\end{proof}

\begin{proof}[Proof of Proposition \ref{proposition1.7}]
If we multiply \eqref{eq2.23} by $-\frac{2}{m}$ and integrate over a closed manifold $M$, we obtain
\begin{equation}
\label{eq2.40}
\begin{split}
0=&\, \int_M \left [ \frac{1}{m}\left ( \frac34 +\frac{1}{m} \right ) |S|^2 -\frac{1}{m}\left ( \frac14 + \frac{1}{m}\right ) |A|^2 +\frac{2}{m} \left ( \frac{2}{m}+\frac14 \right ) |X|^2\divergence X\right . \\
&\, \left . +\frac4m \left ( \frac14 +\frac1m \right )|X|^2 \left ( \lambda +\frac1m |X|^2\right ) -\frac2m \left ( \frac2m +1\right ) (\divergence X)^2\right ] dV.
\end{split}
\end{equation}
On the other hand, if we multiply \eqref{eq2.26} by $\left ( \frac14+\frac2m \right )$ and, assuming that $m\neq -2$, divide out the prefactor $\left ( 1+\frac{2}{m}\right )$, we obtain
\begin{equation}
\label{eq2.41}
\begin{split}
0=&\, \int_M\left [ -\frac2m \left ( \frac14 + \frac2m \right ) |X|^2 \divergence X +\left ( \frac14 +\frac2m \right )(\divergence X)^2\right . \\
&\, \left . -\frac12 \left ( \frac14 +\frac2m\right )|S|^2 \right ] dV
\end{split}
\end{equation}
whenever $m\neq -2$. Add these, and assume that $m\neq -4$ in order to remove a common factor of $\left ( \frac14 +\frac1m\right )$ from the result. We get
\begin{equation}
\label{eq2.42}
\begin{split}
0=&\, \int_M \left [ \left ( \frac1m -\frac12 \right )|S|^2-\frac1m |A|^2 +\left ( 1-\frac4m\right ) (\divergence X)^2 \right . \\
&\, \left . +\frac4m |X|^2 \left ( \lambda +\frac1m |X|^2\right ) \right ] dV\\
=&\, \int_M \left [ \left ( \frac1m -\frac12 \right )|S|^2-\frac1m |A|^2 +\left ( 1-\frac4m\right ) (\divergence X)^2 \right . \\
&\, \left . +\left ( \lambda +\frac2m |X|^2\right )^2 -\lambda^2\right ] dV.
\end{split}
\end{equation}
A simple integration of the last term yields equation \eqref{eq1.7} and proves the proposition.
\end{proof}

\begin{proof}[Proof of Corollary \ref{corollary1.8}]
When $|X|^2=-m\lambda$, equation \eqref{eq1.7} reduces to
\begin{equation}
\label{eq2.43}
0  =\int_M \left [ \left ( \frac12 -\frac1m \right )\left \vert \pounds_X g\right \vert^2 +\frac1m \left \vert dX\right \vert^2 +\left ( \frac4m -1\right ) (\divergence X)^2 \right ] dV .
\end{equation}
For $2<m\le 4$ clearly $\pounds_X g = 0$ and $dX=0$, so $\nabla X=0$. Then $(M,g)$ splits isometrically as claimed \cite{Wylie}. If $m=2$, then we have only that $dX=0$ and $\divergence X=0$, but then we can invoke Corollary \ref{corollary1.5}.(iii) to deduce local conformal flatness and the stated isometric splitting when $n=4$.
\end{proof}

\appendix
\section{Nontrivial examples of quasi-Einstein manifolds with incompressible $X$}\label{appendix}
\setcounter{equation}{0}

\noindent We now exhibit a simple, known class of closed quasi-Einstein manifolds $(M,g,X)$ for any $m>0$ that have incompressible vector field $X$ but are not left-invariant metrics on a homogeneous manifold. Consider a metric defined on a $(2n+1)$-dimensional manifold $M$ of the form~\cite{Kunduri:2012uq}
\begin{equation}
\label{eqA.1}
g = (d\psi + \hat\sigma)^2 + \hat{g}
\end{equation}
where $\hat \sigma$ is a one-form and $\hat{g}$ is a metric defined in a $2n$-dimensional manifold. Let $X = \alpha \partial_\psi$ where $\alpha \geq 0$ is a constant. Obviously, $|X|^2 = \alpha^2$ is then constant. Define $\hat\omega = d \hat \sigma/2$. The quasi-Einstein equation \eqref{eq1.2} reduces to
\begin{equation}
\label{qEexample}
\begin{split}
\hat{R}_{ij} =&\, 2 \hat \omega_{ik}\hat\omega_{j}^{\;k} + \lambda \hat{g}_{ij},\\
\hat {\nabla}^i \hat \omega_{ij} =&\, 0, \\
\hat\omega_{ij}\hat\omega^{ij} =&\, \frac{\alpha^2}{m} + \lambda .
\end{split}
\end{equation}
where $\hat{R}_{ij}$ is the Ricci curvature of the Levi-Civita connection $\hat\nabla$ of $\hat{g}$. Assume that $\hat\omega \neq 0$ and $\alpha^2/m + \lambda >0$ (note that $\lambda <0$ is still permitted). We may choose the transverse space to be a K\"ahler manifold $(B,\hat{g}, \hat{J})$ with complex structure $\hat{J}$. Fixing $\hat\omega$ to be proportional to $\hat{J}$, the second equation of \eqref{qEexample} is automatically satisfied, and the remaining equations are satisfied provided
\begin{equation}
\label{eqA.3}
\begin{split}
\hat \omega =&\, \sqrt{\frac{\alpha^2}{2mn} + \frac{\lambda}{2n}} \hat{J},\\
\hat{R}_{ij} =&\, \left( \frac{\alpha^2}{mn} + \left( 1 + \frac{1}{n}\right)\lambda \right) \hat{g}_{ij}
\end{split}
\end{equation}
This shows that $(B,\hat{g},\hat{J})$ is actually K\"ahler-Einstein.

For $n=1$ this is in fact the only solution to \eqref{qEexample}, namely $(B, \hat{g})$ is ${\mathbb S}^2 \cong \mathbb{CP}^1$ with its suitably normalized round metric. In this case the full quasi-Einstein metric, upon choosing $\psi$ to be periodically identified, will be ${\mathbb S}^3$ with a homogeneously squashed (Berger) metric.

To give examples that are not left-invariant metrics on homogeneous spaces, we now take $n>1$. We may choose $(B,\hat{g}, \hat{J})$ to be any closed K\"ahler-Einstein manifold. For concreteness, consider the following family of K\"ahler metrics, used in the construction of Sasaki-Einstein manifolds \cite{Gauntlett:2004yd, Gauntlett:2004hh}:
\begin{equation}
\label{eqA.4}
\begin{split}
\hat{g} =&\, \frac{x^{n-1} dx^2}{2P(x)} + \frac{2P(x)}{x^{n-1}} (d\phi + \bar{\sigma})^2 + 2 x \bar{g}, \\
\hat{J} =&\, d \left [ x (d\phi + \bar\sigma)\right ].
\end{split}
\end{equation}
with $\bar{g}$ a K\"ahler-Einstein metric on some $(2n-2)$-dimensional base space $K$, normalized so that $\ric (\bar{g}) = 2n \bar{g}$, and with K\"ahler form $\bar{\hat{J}} = d\bar\sigma/2$. The choice
\begin{equation}
\label{eqA.5}
P(x) = x^n  - \left(\frac{\alpha^2}{m(n+1)} + \lambda \right)\frac{x^{n+1}}{n} + c,
\end{equation}
where $c$ is a constant, guarantees that $\hat{g}$ is itself Einstein with the correct scaling. Note that if one chooses $(K,\bar{g})$ to be $\mathbb{CP}^{n-1}$ with its homogeneous Fubini-Study metric, \eqref{eqA.5} is cohomogeneity-one, but in general, it need only have one Killing field, namely $\partial_\phi$. The construction described above produces a local quasi-Einstein metric
\begin{equation}
\label{eqA.6}
\begin{split}
g =&\, (d\psi + \beta x (d\phi + \bar\sigma))^2 + \hat{g},\\
X =&\, \alpha \frac{\partial}{\partial \psi},\\
\beta =&\, \sqrt{\frac{2 \alpha^2}{m n} + \frac{2\lambda}{n}}.
\end{split}
\end{equation}
Under suitable restrictions on the parameters, the local metric $g$ extends to a smooth metric on a compact manifold \cite{Gauntlett:2004yd, Gauntlett:2004hh}. Potential singularities may occur at $x=0$ and roots of $P(x)$. The global analysis of metrics of this type was carried out in \cite{Kunduri:2012uq} which shows that (for each choice of $\alpha$)  there exists a countably infinite family of smooth quasi-Einstein manifolds $M^{p,q}$ labelled by a pair of integers $(p,q)$ satisfying $1 < I p / q < 2$ where $I$ is the Fano index of $K$. Assuming $p$ and $q$ are co-prime, then $M^{p,q}$ is a lens space bundle over $K$; for $n=2$, $M^{p,q} \cong {\mathbb S}^2 \times {\mathbb S}^3$.

If one begins the above construction by taking $(B, \hat{g}, \hat{J})$ to have no continuous globally defined isometries (e.g., if it is negative Einstein), the construction will produce a quasi-Einstein manifold with a single continuous isometry generated by $X$. In this sense, Theorem \ref{theorem1.1} is sharp.


\begin{thebibliography}{99}
\bibitem{ABR} C Aquino, A Barros, and E Ribeiro Jr, \emph{Some applications of the Hodge-de Rham decomposition to Ricci solitons}, Results Math 60 (2011) 245--254.
\bibitem{BGKW} E Bahuaud, S Gunasekaran, HK Kunduri, and E Woolgar, \emph{Static near-horizon geometries and rigidity of quasi-Einstein manifolds}, Lett Math Phys 112:116, (2022) 1--16.
\bibitem{BR} A Barros and E Ribeiro Jr, \emph{Integral formulae on quasi-Einstein manifolds and applications} Glasgow Math J 54 (2012) 213--223.
\bibitem{Case} JS Case, \emph{On the nonexistence of quasi-Einstein metrics}, Pac J Math 248 (2010) 277--284.
\bibitem{CSW} JS Case, Y-J Shu, and G Wei, \emph{Rigidity of quasi-Einstein metrics}, Diff Geom Appl 29 (2011) 93--100.
\bibitem{CRT} PT Chru\'sciel, HS Reall, and P Tod, \emph{On non-existence of static vacuum black holes with degenerate components of the event horizon}, Class Quantum Gravit 23 (2006) 549--554.
\bibitem{DL} M Dunajski and J Lucietti, \emph{Intrinsic rigidity of extremal horizons}, preprint [arxiv:2306.17512].
\bibitem{FFT} AA Freitas Filho, and K Tenenblat, \emph{On generalized quasi-Einstein manifolds}, J Geom Phys 178 (2022) 104562.
\bibitem{Gauntlett:2004yd}
JP Gauntlett, D Martelli, J Sparks, and D Waldram,
\emph{Sasaki-Einstein metrics on ${\mathbb S}^2\times {\mathbb S}^3$},
Adv Theor Math Phys \textbf{8} (2004) 711--734.

\bibitem{Gauntlett:2004hh}
JP Gauntlett, D Martelli, JF Sparks, and D Waldram,
\emph{A New infinite class of Sasaki-Einstein manifolds},
Adv Theor Math Phys \textbf{8} (2004) 987--1000.

\bibitem{KK} D-S Kim and Y H Kim, \emph{Compact Einstein warped product spaces with nonpositive scalar curvature}, Proc Amer Math Soc 131 (2003) 2573--2576.

\bibitem{Kunduri:2012uq}
HK Kunduri and J Lucietti,
\emph{Extremal Sasakian horizons},
Phys Lett B\textbf{713} (2012) 308--312.

\bibitem{KL} HK Kunduri and J Lucietti, \emph{Classification of near-horizon geometries of extremal black holes}, Living Reviews in Relativity 16--8 (2013).
\bibitem{Lim} A Lim, \emph{Locally homogeneous non-gradient quasi-Einstein 3-manifolds}, Adv  Geom 22 (2022) 79--93.

\bibitem{Petersen} P Petersen, \emph{Riemannian geometry}, second edition, Graduate Texts in Mathematics 171 (Springer, New York, 2006).
\bibitem{Wylie} W Wylie, Rigidity of compact static near-horizon geometries with
negative cosmological constant, Lett Math Phys 113:29 (2023) 1--5.
\end{thebibliography}
\end{document}